\documentclass[reqno,draft]{amsart}

\theoremstyle{plain}
\newtheorem{lemma}{Lemma}
\newtheorem{theorem}[lemma]{Theorem}

\theoremstyle{remark}

\def\aa{\alpha}

\def\Dd{\Delta}

\def\Om{\Omega}

\def\pp{\partial}

\begin{document}

\title[Regularity for the Three-dimensional Navier--Stokes Equations]
{Global  regularity criterion  for the  $3D$
Navier--Stokes equations involving one entry of the velocity gradient tensor}

\date{May 24, 2010}

\author[C. Cao]{Chongsheng Cao}
\address[C. Cao]
{Department of Mathematics  \\
Florida International University  \\
Miami, FL 33199, USA}
\email{caoc@fiu.edu}

\author[E.S. Titi]{Edriss S. Titi}
\address[E.S. Titi]
{Department of Mathematics \\
and  Department of Mechanical and  Aerospace Engineering \\
University of California \\
Irvine, CA  92697-3875, USA \\
{\bf also}  \\
Department of Computer Science and Applied Mathematics \\
Weizmann Institute of Science  \\
Rehovot 76100, Israel} \email{etiti@math.uci.edu  and edriss.titi@weizmann.ac.il}

\begin{abstract}
In this paper we  provide a sufficient condition,
in terms of only one of the nine entries of the
gradient tensor, i.e., the Jacobian matrix of the
velocity vector field,
 for the global regularity  of strong solutions to the three--dimensional
Navier--Stokes equations in the whole space, as
well as for the case of periodic boundary
conditions.
\end{abstract}

\maketitle

\vskip 0.125in

AMS Subject Classifications: 35Q35, 65M70

Key words: Three-dimensional Navier--Stokes equations, Regularity
criterion for Navier--Stokes equations, global regularity,

\section{Introduction}   \label{S-1}

The three-dimensional  Navier--Stokes equations (NSE) of viscous
incompressible fluid read:
\begin{eqnarray}
&&\hskip-.8in
\frac{\pp u}{\pp t} - \nu  \Dd u  + ( u \cdot \nabla ) u  + \nabla p = 0,  \label{EQ-1}  \\
&&\hskip-.8in
\nabla \cdot u  =0,   \label{EQ-2}
\\&&\hskip-.8in
u(x_1,x_2,x_3,0) = u_0 (x_1,x_2,x_3),  \label{EQ-3}
\end{eqnarray}
where $u=(u_1, u_2, u_3)$, the velocity field, and $p$, the
pressure, are the unknowns, and   $\nu > 0$, the viscosity, is
given. We set $\nabla_h  = (\pp_{x_1}, \pp_{x_2})$ to be the
horizontal gradient operator and $\Dd_h = \pp_{x_1}^2 +\pp_{x_2}^2$
the horizontal Laplacian, while $\nabla$ and $\Dd$ are the usual
gradient and the Laplacian operators, respectively. In this paper we consider finite energy
solutions of   the
system (\ref{EQ-1})--(\ref{EQ-3}) in the whole space $\mathbb{R}^3$, that decay at infinity.
However, we remark  that one can apply our proof, nearly
line by line, to establish same result for the
three--dimensional  Navier--Stokes equations in a periodic domain.

The question of global regularity for the $3D$
Navier--Stokes equations is a major challenging
problem in applied analysis. Over the years there
has been an intensive work by many authors
attacking this problem (see, e.g.,  \cite{CP95},
\cite{CF88}, \cite{DJ95}, \cite{LADY69},
\cite{LADY03}, \cite{LR02}, \cite{LL69},
\cite{PL96}, \cite{SO01A}, \cite{TT84},
\cite{TT95}, \cite{TT00} and references therein).
It is well-known that the 2D Navier--Stokes
equations have a unique weak and strong solutions
which exist globally in time (cf., for example,
\cite{CF88}, \cite{LADY69}, \cite{SO01A},
\cite{TT84}, \cite{TT95}). In the $3D$ case, the
weak solutions are known to exist globally in time.
But, the uniqueness, regularity, and continuous
dependence on initial data for weak solutions are
still open problems. Furthermore, strong solutions
in the $3D$ case are known to exist for a short
interval of time whose length depends on the
physical data of the initial--boundary value
problem. Moreover, this strong solution is known to
be unique and depend continuously on the initial
data (cf., for example, \cite{CF88},
\cite{LADY69},\cite{SO01A}, \cite{TT84}).

Starting from  the pioneer works of Prodi \cite{PG59} and of Serrin
\cite{SJ62}, many articles were dedicated for providing sufficient
conditions for the global regularity of the $3D$ Navier--Stokes
equations (for details see, for example,  the survey papers
 \cite{LADY03}, \cite{TT00} and references therein).
 Most recently, there has been
some progress along these lines (see, for example, \cite{BL02},
\cite{BG02}, \cite{ESV03}, \cite{HK64}, \cite{GI86}, \cite{GM85}, \cite{KT84},
\cite{SO01}, \cite{SO02}, and references therein) which states,
roughly speaking,  that a strong solution $u$ exists on the time
interval $[0, T]$ for as long as
\begin{equation}
u \in L^{p} ([0, T], L^{q}), \quad\text{with}\quad
\frac{2}{p}+\frac{3}{q}=1, \quad \text{for} \quad q \geq 3.
\label{SP}
\end{equation}
Moreover, that has also been some works dedicated to the study
the global regularity of the $3D$ Navier--Stokes equations by
providing some sufficient conditions on the pressure (cf. e.g.,
\cite{BG02}, \cite{CT08}, \cite{CL01}, \cite{DV00}, \cite{Ku06}, \cite{SS02}, \cite{ZY05}). In
addition, some other sufficient regularity conditions were
established in terms of only one component of the velocity field of
the $3D$ NSE on the whole space $\mathbb{R}^3$ or under periodic
boundary conditions (cf. e.g.,  \cite{CT08},  \cite{HE02}, \cite{KZ05},
\cite{PO03}, \cite{PZ10}, \cite{ZY02}).

\vskip0.1in

We denote by $L^q$ and $H^m $ the usual $L^q-$Lebesgue and
Sobolev spaces, respectively (cf. \cite{AR75}), and by
\begin{equation}
\| \phi\|_q =  \left(  \int_{\mathbb{R}^3} |\phi(x)|^q \; dx_1dx_2dx_3
\right)^{\frac{1}{q}},   \qquad  \mbox{ for every $\phi \in
L^q$}.
 \label{LQ}
\end{equation}
We set
\begin{eqnarray*}
 \mathcal{V} &=&  \left\{ \phi: \mbox{the three-dimensional
 vector valued $C^{\infty}_0$ functions and}~
 \nabla \cdot \phi = 0  \right\},
\end{eqnarray*}
which will form the space of test functions.  Let $H$ and $V$ be the
closure spaces of $\mathcal{V}$ in $L^2$ under $L^2-$topology,
and  in $H^1$ under $H^1-$topology, respectively. Let $u_0 \in
H$, we say $u$ is a Leray--Hopf weak solution to the system
(\ref{EQ-1})--(\ref{EQ-3}) on the interval $[0, T]$ with initial
value $u_0$ if $u$ satisfies the following three conditoins:

\begin{itemize}

\item[(1)]
$u  \in C_w([0, T], H) \cap L^2([0, T], V),  $ and
$ \pp_t u  \in L^1([0, T], V^{\prime}),$
where $V^{\prime}$ is the dual space of $V$;

\item[(2)] the weak formulation of the NSE:
\begin{eqnarray*}
&&\hskip-0.35in
\int_{\mathbb{R}^3}  u(x, t) \cdot \phi(x, t) \, dx
-\int_{\mathbb{R}^3}  u(x, t_0) \cdot \phi(x, t_0)  \, dx  \nonumber  \\
&&\hskip-0.25in
=  \int_{t_0}^t
\int_{\mathbb{R}^3} \left[ u(x, t) \cdot \left( \phi_t(x, t)
+ \nu \Dd \phi(x, t)  \right) \right]
\, dx \; ds  \nonumber \\
&&\hskip-0.25in + \int_{t_0}^t \int_{\mathbb{R}^3}  \left[  ( u(x, t) \cdot
\nabla) \phi(x, t)  \right]  \cdot u(x, t)   \, dx,    \label{WEAK}
\end{eqnarray*}
for every test function $\phi \in C^{\infty} ([0, T], \mathcal{V}),$ and
for almost every $t$, $t_0\in [0,T]$;

\item[(3)] the energy inequality:
\begin{eqnarray}
&&\hskip-.68in
\|u(t)\|_2^2 + \nu \int_{t_0}^t
  \| \nabla u (s) \|_2^2 ~ ds
 \leq   \|u(t_0)\|_2^2,  \label{ENG}
\end{eqnarray}
for every $t$ and almost every $t_0$.
\end{itemize}
Moreover, if $u_0\in V$, a weak solution is called strong solution of
(\ref{EQ-1})--(\ref{EQ-3}) on $[0,T]$ if, in addition, it
satisfies
\begin{eqnarray*}
&&u  \in C([0,T], V) \cap L^2([0,T], H^2), ~\mbox{and}~ \pp_t
u \in L^2([0, T], H).
\end{eqnarray*}
In this case, one also has energy equality in (\ref{ENG}) instead of
inequality, and the equality holds for every $t_0$.

\vskip0.1in

In this paper, we provide sufficient conditions, in terms
of only one of the nine components of the gradient of velocity field, i.e.,
the velocity Jacobian matrix, that guarantee the
global regularity of the $3D$ NSE. Specifically, if $u_0 \in V$, and
if for some $T >0 $ and some $k, j,$ with $1 \leq k, j \leq 3,$ we have
\begin{eqnarray}
&&
\frac{\pp u_j}{\pp x_k} \in L^{\beta}([0, T], L^{\aa}(\mathbb{R}^3)); \quad \mbox{when }  k\not= j, \mbox{and where  }
\aa > 3, 1 \leq \beta < \infty,
\mbox{and } \frac{3}{\aa} + \frac{2}{\beta}  < \frac{\aa+3}{2\aa},  \label{CON_1} \\
&&  \mbox{or} \nonumber  \\
&& \frac{\pp u_j}{\pp x_j}  \in L^{\beta} ([0, T], L^{\aa}(\mathbb{R}^3)); \quad \mbox{where  } \aa > 2, 1 \leq \beta < \infty,
\mbox{and } \frac{3}{\aa} + \frac{2}{\beta}  < \frac{3(\aa+2)}{4\aa},
  \label{CON_2}
\end{eqnarray}
where $u=(u_1, u_2, u_3)$ is a weak solution with the initial datum
$u_0$ on $[0, T]$,  then $u$ is a strong solution of the $3D$ Navier--Stokes
equations which exists on the interval $[0,T]$. Moreover, $u$ is the
only weak and strong solution on the interval $[0,T]$ with the
initial datum $u_0$. In particular, if~(\ref{CON_1}) or (\ref{CON_2}) holds
for all $T>0$, then there is a unique global (in time)
 strong solution for the $3D$ NSE with the initial datum $u_0$.

\vskip0.1in

For convenience, we recall the following version of the three-dimensional Sobolev
and Ladyzhenskaya inequalities in the whole space $\mathbb{R}^3$
(see, e.g., \cite{AR75}, \cite{CF88},
\cite{GA94}, \cite{LADY}). There exists a positive constant $C_r$
such that
\begin{eqnarray}
&&\hskip-.68in \| \psi \|_{r} \leq C_r  \| \psi
\|_{2}^{\frac{6-r}{2r}}  \|\pp_{x_1} \psi
\|_2^{\frac{r-2}{2r}} \; \|\pp_{x_2} \psi
\|_2^{\frac{r-2}{2r}}\; \|\pp_{x_3} \psi
\|_2^{\frac{r-2}{2r}}  \nonumber  \\
&&\hskip-.68in
\leq C_r \| \psi
\|_{2}^{\frac{6-r}{2r}} \; \| \psi
\|_{H^1(\mathbb{R}^3)}^{\frac{3(r-2)}{2r}},  \label{SI1}
\end{eqnarray}
for every $\psi\in H^1(\mathbb{R}^3)$ and every $r \in  [2, 6].$
Observe that in case of periodic boundary conditions, one would have instead of (\ref{SI1}) the following inequality
\begin{eqnarray}
&&\hskip-.68in \| \psi \|_{r} \leq C_r  \| \psi
\|_{2}^{\frac{6-r}{2r}} \left( \|\pp_{x_1} \psi
\|_2+ \|\psi\|_2 \right)^{\frac{r-2}{2r}} \left( \|\pp_{x_2} \psi
\|_2+ \|\psi\|_2 \right)^{\frac{r-2}{2r}}\left( \|\pp_{x_3} \psi
\|_2+ \|\psi\|_2 \right)^{\frac{r-2}{2r}}  \nonumber  \\
&&\hskip-.68in
\leq C_r \| \psi
\|_{2}^{\frac{6-r}{2r}} \; \| \psi
\|_{H^1(\Om)}^{\frac{3(r-2)}{2r}},  \label{SI1-1}
\end{eqnarray}
for every $\psi\in H^1(\Om)$ and every $r \in  [2, 6].$ Here, $\Om$ is the periodic box $[0, L]^3$.
We remark that one can apply inequality (\ref{SI1-1}) instead of  inequality (\ref{SI1}), and
the methods presented in this paper to
establish the same results in the case of periodic boundary conditions.
The details of the proof in the periodic case are omitted.

\vskip0.2in

\section{The Main Result} \label{S-2}

In this section we will prove  our main result, which states that
the strong solution to  system (\ref{EQ-1})--(\ref{EQ-3}) exists
 on the interval $[0,T]$ provided  the assumption (\ref{CON_1}) or (\ref{CON_2}) holds.

\begin{theorem} \label{T-MAIN}
Let $u_0 \in V,$ and let $u=(u_1,u_2, u_3)$ be a Leray--Hopf weak
solution to $3D$ NSE,  system (\ref{EQ-1})--(\ref{EQ-3}), with the initial value
$u_0$. Let $T>0$, and suppose that, for some  $k, j,$ with $1 \leq k, j \leq 3,$   $u$ satisfies the condition
(\ref{CON_1}) or (\ref{CON_2}), namely,
\begin{eqnarray}
&&
\int_0^T \left\| \frac{\pp u_j(s)}{\pp x_k} \right\|_{\aa}^{\beta} \; ds \leq M; \quad \mbox{when }  k\not= j, \mbox{and where  }
  \aa > 3, 1 \leq \beta < \infty,
\mbox{and } \frac{3}{\aa} + \frac{2}{\beta}  \leq \frac{\aa+3}{2\aa},    \label{CON-1}
\end{eqnarray}
or
\begin{eqnarray}
&&
\int_0^T \left\| \frac{\pp u_j(s)}{\pp x_j} \right\|_{\aa}^{\beta} \; ds \leq M; \quad \mbox{where }
  \aa > 2, 1 \leq \beta < \infty,
\mbox{and } \frac{3}{\aa} + \frac{2}{\beta}  \leq \frac{3(\aa+2)}{4\aa},    \label{CON-2}
\end{eqnarray}
for some $M>0.$ Then  $u$ is a strong solution of $3D$ NSE, system {\em
(\ref{EQ-1})--(\ref{EQ-3})}, on the interval $[0, T]$. Moreover, it
is the only weak solution on $[0, T]$ with the initial datum $u_0$.
\end{theorem}

\vskip0.05in

Before we prove the main Theorem \ref{T-MAIN}, we show the following two lemmas.

\begin{lemma}  \label{LEMMA-1}
\begin{eqnarray}
&& \left| \int_{\mathbb{R}^3} \phi \; f \; g \;  dx_1dx_2dx_3  \right| \leq
C \| \phi\|_{2}^{\frac{r-1}{r}}  \| \pp_{x_1} \phi\|_{\frac{2}{3-r}}^{1/r}
\|  f \|_{2}^{\frac{r-2}{r}} \; \|\pp_{x_2} f \|_2^{1/r} \; \|\pp_{x_3} f \|_2^{1/r} \; \| g \|_{2},
   \label{LEM-1}
\end{eqnarray}
where $2< r < 3.$
\end{lemma}

\begin{proof} Observe, first, that it is enough to prove the inequality for functions
$\phi, f, g \in C^{\infty}_0(\mathbb{R}^3)$ and then passing to the limit using a density argument.

\begin{eqnarray*}
&& \left| \int_{\mathbb{R}^3} \phi \; f \; g \; dx_1dx_2dx_3  \right|
\leq C \int_{\mathbb{R}^2} \left[
\max_{x_1} |\phi| \; \left( \int_{\mathbb{R}}
f^2 dx_1\right)^{1/2}
\left( \int_{\mathbb{R}} g^2 dx_1\right)^{1/2} \right] \; dx_2dx_3  \\
&&  \leq  C \left[ \int_{\mathbb{R}^2} \left( \max_{x_1} |\phi| \right)^{r} \; dx_2dx_3 \right]^{1/r}
\left[ \int_{\mathbb{R}^2} \left( \int_{\mathbb{R}} f^2 dx_1\right)^{\frac{r}{r-2}} \; dx_2dx_3 \right]^{\frac{r-2}{2r}}
\left( \int_{\mathbb{R}^3} g^2 dx_1dx_2dx_3\right)^{1/2}   \\
&&
\leq C  \left[ \int_{\mathbb{R}^3}   |\phi|^{r-1}  |\pp_{x_1} \phi |  \; dx_1dx_2dx_3 \right]^{1/r}
\left[ \int_{\mathbb{R}} \left( \int_{\mathbb{R}^2} f^{\frac{2r}{r-2}} dx_2dx_3 \right)^{\frac{r-2}{r}} \;
dx_1 \right]^{\frac{1}{2}}
\| g \|_2  \\
&&  \leq
C \| \phi\|_{2}^{\frac{r-1}{r}}  \| \pp_{x_1} \phi\|_{\frac{2}{3-r}}^{1/r}
\|  f \|_{2}^{\frac{r-2}{r}} \; \|\pp_{x_2} f \|_2^{1/r} \; \|\pp_{x_3} f \|_2^{1/r} \; \| g \|_{2}.
\end{eqnarray*}

\end{proof}

\begin{lemma}   \label{LEMMA-2}
\begin{eqnarray}
&& \left| \int_{\mathbb{R}^3} \phi \; f \; g \; dx_1dx_2dx_3 \right|  \leq
C \| \phi\|_{2}^{\frac{r-1}{r}}  \| \pp_{x_3} \phi\|_{\frac{2}{3-r}}^{1/r}
\|  f \|_{2}^{\frac{r-2}{r}} \; \|\pp_{x_1} f \|_2^{1/r} \; \|\pp_{x_2} f \|_2^{1/r} \; \| g \|_{2},
   \label{LEM-2}
\end{eqnarray}
where $2< r < 3.$
\end{lemma}

\begin{proof} Here, again, it is enough to prove the inequality for functions
$\phi, f, g \in C^{\infty}_0(\mathbb{R}^3)$.
\begin{eqnarray*}
&& \left| \int_{\mathbb{R}^3} \phi \; f \; g \; dx_1dx_2dx_3 \right|  \leq C \int_{\mathbb{R}^2}
\left[  \max_{x_3} |\phi| \; \left( \int_{\mathbb{R}} f^2 dx_3\right)^{1/2}
\left( \int_{\mathbb{R}} g^2 dx_3\right)^{1/2}  \right] \; dx_1dx_2  \\
&&  \leq  C \left[ \int_{\mathbb{R}^2} \left( \max_{x_3} |\phi| \right)^{r} \; dx_1dx_2 \right]^{1/r}
\left[ \int_{\mathbb{R}^2} \left( \int_{\mathbb{R}} f^2 dx_3\right)^{\frac{r}{r-2}} \; dx_1dx_2 \right]^{\frac{r-2}{2r}}
\left( \int_{\mathbb{R}^3} g^2 dx_1dx_2dx_3\right)^{1/2}   \\
&&
\leq C  \left[ \int_{\mathbb{R}^3}   |\phi|^{r-1}  |\pp_{x_3} \phi |  \; dx_1dx_2dx_3 \right]^{1/r}
\left[ \int_{\mathbb{R}} \left( \int_{\mathbb{R}^2} f^{\frac{2r}{r-2}} dx_1dx_2 \right)^{\frac{r-2}{r}} \; dx_3 \right]^{\frac{1}{2}}
\| g \|_2  \\
&&  \leq
C \| \phi\|_{2}^{\frac{r-1}{r}}  \| \pp_{x_3} \phi\|_{\frac{2}{3-r}}^{1/r}
\|  f \|_{2}^{\frac{r-2}{r}} \; \|\pp_{x_1} f \|_2^{1/r} \; \|\pp_{x_2} f \|_2^{1/r} \; \| g \|_{2}.
\end{eqnarray*}

\end{proof}

\noindent
{\bf Proof of the Theorem \ref{T-MAIN}.}
 Without loss of generality, we will assume that $j=3$ and $k=1$ in (\ref{CON-1}) and (\ref{CON-2}),
 namely,
\begin{eqnarray}
&&
\int_0^T \left\| \frac{\pp u_3(s)}{\pp x_1} \right\|_{\aa}^{\beta} \; ds \leq M; \quad \mbox{where  }
  \aa > 3, 1 \leq \beta < \infty,
\mbox{and } \frac{3}{\aa} + \frac{2}{\beta}  \leq \frac{\aa+3}{2\aa},    \label{CONN-1}
\end{eqnarray}
or,
\begin{eqnarray}
&&
\int_0^T \left\| \frac{\pp u_3(s)}{\pp x_3} \right\|_{\aa}^{\beta} \; ds \leq M; \quad \mbox{where }
  \aa > 2, 1 \leq \beta < \infty,
\mbox{and } \frac{3}{\aa} + \frac{2}{\beta}  \leq \frac{3(\aa+2)}{4\aa}.    \label{CONN-2}
\end{eqnarray}

It is well-known that there exists a global in time
Leray--Hopf weak solution to the $3D$ NSE, the
system (\ref{EQ-1})--(\ref{EQ-3}), in whole space
$\mathbb{R}^3$ if $u_0 \in H$ (see, e.g.,
\cite{CF88}, \cite{LADY03}, \cite{LR34},
\cite{PL96}, \cite{SO01A}, \cite{TT95}). It is also
well-known that there exists a unique strong
solution for a short time interval if $u_0 \in V$.
In addition, this strong solution is the only weak
solution, with the initial datum $u_0$, on the
maximal interval of existence of the strong
solution.

Suppose that $u$ is the strong solution with the initial value $u_0 \in
V$ such that $u \in C([0, \mathcal{T}^*), V) \cap L^2 ([0,
\mathcal{T}^*), H^2),$ where $[0,\mathcal{T}^*)$ is the maximal
interval of existence of the unique strong solution. If
$\mathcal{T}^* \ge T$ then there is nothing to prove. If, on the
other hand, $\mathcal{T}^* < T$ our strategy is to show that the
$H^1$ norm of this strong solution is bounded uniformly in time over the interval
$[0,\mathcal{T}^*)$, provided condition (\ref{CONN-1}) or (\ref{CONN-2}) is valid. As a
result the interval $[0, \mathcal{T}^*)$ can not be a maximal interval of
existence, and consequently $\mathcal{T}^* \ge T$. Which will
conclude our proof.

{F}rom now on we focus on the strong solution, $u$, on its maximal
interval of existence $[0, \mathcal{T}^*),$ where  we assume that
$\mathcal{T}^* < T.$  As we have observed earlier the strong
solution $u$ will also be the only weak solution on the interval
$[0,\mathcal{T}^*)$. Therefore, by the energy inequality (\ref{ENG}),
for Leray--Hopf weak solutions, we have (see, e.g.,
\cite{CF88}, \cite{SO01A}, \cite{TT84} or \cite{TT95} for details)
\begin{eqnarray}
&&\hskip-.68in \|u(t)\|_2^2 +\nu \int_0^t  \| \nabla u (s)\|_2^2 \; ds
 \leq K_1,   \label{K-1}
\end{eqnarray}
for all $t\geq 0$, where
\begin{eqnarray}
&&\hskip-.68in
K_1=  \|u_0\|_2^2.  \label{K1}
\end{eqnarray}

Next, let us show that the $H^1$ norm of the strong solution $u$ is bounded on interval $[0, \mathcal{T}^*)$.


\subsection{$\|\nabla_h u\|_2$ estimates}

First we obtain some estimates of the horizontal gradient.
Taking the  inner product of the  equation (\ref{EQ-1}) with $-\Dd_h u$
 in $L^2$, we obtain
\begin{eqnarray*}
&&\hskip-.22in \frac{1}{2} \frac{d \|\nabla_h u \|_2^2  }{d t} + \nu
 \|\nabla_h \nabla u \|_2^2  \\
&&\hskip-.2in = \int_{\mathbb{R}^3} \left[  (u \cdot \nabla) u
\right] \cdot \Dd_h u  \; dx_1dx_2dx_3.
\end{eqnarray*}
By integration by parts few  times, and using the incompressibility
condition~(\ref{EQ-2}), we get
\begin{eqnarray*}
&&\hskip-.265in    -\int_{\mathbb{R}^3} (u \cdot \nabla) u \cdot \Dd_h u \; dx_1dx_2dx_3 =
\int_{\mathbb{R}^3} \sum_{k,j=1}^3 \sum_{l=1}^2 \frac{\pp u_k}{\pp x_j}\frac{\pp u_j}{\pp x_l}\frac{\pp
u_k}{\pp x_l} \; dx_1dx_2dx_3   \\
&&\hskip-.265in =  \int_{\mathbb{R}^3} \left\{ \left(\frac{\pp u_1}{\pp
x_1}\right)^3+\left(\frac{\pp u_2}{\pp x_2}\right)^3 +  \left(\frac{\pp u_1}{\pp x_1}+\frac{\pp u_2}{\pp
x_2}\right)  \left[ \left(\frac{\pp u_1}{\pp x_2}\right)^2 +\left(\frac{\pp u_2}{\pp x_1}\right)^2
+ \frac{\pp u_1}{\pp x_2}\;\frac{\pp u_2}{\pp x_1} \right]  \right.  \\
&&\hskip-.0015in
 +  \left. \sum_{k, l=1}^2 \frac{\pp u_k}{\pp x_3}\frac{\pp
u_3}{\pp x_l}\frac{\pp u_k}{\pp x_l} + \sum_{j,l=1}^2  \frac{\pp u_3}{\pp x_j}\frac{\pp u_j}{\pp x_l}\frac{\pp
u_3}{\pp x_l} + \sum_{l=1}^2  \frac{\pp u_3}{\pp x_3}\frac{\pp u_3}{\pp x_l}\frac{\pp
u_3}{\pp x_l}  \right\} \; dx_1dx_2dx_3    \\
&&\hskip-.265in =    \int_{\mathbb{R}^3} \left\{ - \; \frac{\pp u_3}{\pp x_3} \left[
\left(\frac{\pp u_1}{\pp
x_1}\right)^2+\left(\frac{\pp u_2}{\pp x_2}\right)^2 -\frac{\pp u_1}{\pp x_1}\frac{\pp u_2}{\pp x_2}
+  \left(\frac{\pp u_1}{\pp x_2}\right)^2 +\left(\frac{\pp u_2}{\pp x_1}\right)^2 +
\frac{\pp u_1}{\pp x_2}\;\frac{\pp u_2}{\pp x_1}
 \right] \right.   \\
&&\hskip-.0015in
 +  \left. \sum_{l=1}^2  \frac{\pp
u_3}{\pp x_l} \left[ \sum_{k=1}^2 \frac{\pp u_k}{\pp x_3}
 \frac{\pp u_k}{\pp x_l} + \sum_{k=1}^2  \frac{\pp u_3}{\pp x_k}\frac{\pp u_k}{\pp x_l}
  +   \frac{\pp u_3}{\pp x_3}\frac{\pp
u_3}{\pp x_l} \right] \right\} \; dx_1dx_2dx_3    \\
&&\hskip-.265in
\leq     C \int_{\mathbb{R}^3}  |u_3| \, |\nabla u|~|\nabla_h \nabla u|
  \; dx_1dx_2dx_3.
\end{eqnarray*}
Next, we will estimate the right-hand side of the above inequality using either Lemma \ref{LEMMA-1} or
Lemma \ref{LEMMA-2}. Each will be used for dealing with either one of the conditions (\ref{CONN-1}) or (\ref{CONN-2}).
On the one hand by applying (\ref{LEM-1}),   with $\phi=|u_3|, f = |\nabla u|, g =|\nabla_h \nabla u|,$
and $ r = \frac{3\aa-2}{\aa}$,  we get
\begin{eqnarray}
&&\hskip-.268in
\int_{\mathbb{R}^3} |u_3| \, |\nabla u | \, |\nabla_h \nabla u|
 \; dx_1dx_2dx_3  \nonumber \\
 &&\hskip-.268in
 \leq
 C \| u_3\|_{2}^{\frac{2(\aa-1)}{3\aa-2}}  \| \pp_{x_1} u_3\|_{\aa}^{\frac{\aa}{3\aa-2}}
\|  \nabla u\|_{2}^{\frac{\aa-2}{3\aa-2}} \; \|\pp_{x_2} \nabla u \|_2^{\frac{\aa}{3\aa-2}} \;
\|\pp_{x_3} \nabla u \|_2^{\frac{\aa}{3\aa-2}} \; \| \nabla \nabla_h u \|_{2}   \nonumber  \\
&&\hskip-.268in
 \leq
 C \| u_3\|_{2}^2 \| \pp_{x_1} u_3\|_{\aa}^{\frac{\aa}{\aa-1}}
\|  \nabla u\|_{2}^{\frac{\aa-2}{\aa-1}} \;
\|\pp_{x_3} \nabla u \|_2^{\frac{\aa}{\aa-1}} + \frac{\nu}{2}\; \| \nabla \nabla_h u \|_{2}^2.
\label{EST-NN-1}
\end{eqnarray}
On the other hand by applying (\ref{LEM-2}),   with $\phi=|u_3|, f = |\nabla u|, g =|\nabla_h \nabla u|,$ we obtain
\begin{eqnarray}
&&\hskip-.268in
\int_{\mathbb{R}^3} |u_3| \, |\nabla u | \, |\nabla_h \nabla u|
 \; dx_1dx_2dx_3   \nonumber  \\
 &&\hskip-.268in
 \leq
 C \| u_3\|_{2}^{\frac{2(\aa-1)}{3\aa-2}}  \| \pp_{x_3} u_3\|_{\aa}^{\frac{\aa}{3\aa-2}}
\|  \nabla u \|_{2}^{\frac{\aa-2}{3\aa-2}} \; \|\pp_{x_1} \nabla u \|_2^{\frac{\aa}{3\aa-2}} \;
\|\pp_{x_2} \nabla u \|_2^{\frac{\aa}{3\aa-2}} \; \| \nabla \nabla_h u \|_{2} \nonumber   \\
&&\hskip-.268in
 \leq
 C \| u_3\|_{2}^{\frac{4(\aa-1)}{\aa-2}}  \| \pp_{x_3} u_3\|_{\aa}^{\frac{2\aa}{\aa-2}}
\|  \nabla u \|_{2}^{2} + \frac{\nu}{2}\; \| \nabla \nabla_h u \|_{2}^2.   \label{EST-NN-2}
\end{eqnarray}
In case we use (\ref{EST-NN-1})  we obtain
\begin{eqnarray}
&&\hskip-.22in  \frac{d \|\nabla_h u \|_2^2 }{d t} + \nu
 \|\nabla_h \nabla u\|_2^2
\leq
C \| u_3\|_{2}^2  \| \pp_{x_1} u_3\|_{\aa}^{\frac{\aa}{\aa-1}}
\|  \nabla u\|_{2}^{\frac{\aa-2}{\aa-1}} \;
\|\pp_{x_3} \nabla u \|_2^{\frac{\aa}{\aa-1}};
\label{EST-NN-1-1}
\end{eqnarray}
Alteratively, if we use (\ref{EST-NN-2}), we obtain
\begin{eqnarray}
&&\hskip-.32in  \frac{d \|\nabla_h u \|_2^2 }{d t} + \nu
 \|\nabla_h \nabla u\|_2^2
\leq  C \| u_3\|_{2}^{\frac{4(\aa-1)}{\aa-2}}  \| \pp_{x_3} u_3\|_{\aa}^{\frac{2\aa}{\aa-2}}
\|  \nabla u \|_{2}^{2}.
\label{EST-NN-1-2}
\end{eqnarray}
Therefore, integrating (\ref{EST-NN-1-1}) and using H\"{o}lder inequality and applying (\ref{K-1}) we get
\begin{eqnarray}
&&\hskip-.308in
\|\nabla_h u (t)\|_2^2 + \nu \int_0^t \|\nabla_h \nabla u (s) \|_2^2  \; ds  \nonumber  \\
&&\hskip-.308in
\leq \|\nabla_h u_0\|_2^2 + C  \left( \int_0^t\|\pp_{x_1} u_3(s)\|_{\aa}^{\frac{2{\aa}}{{\aa}-2}}  \|\nabla u (s)\|_2^2\; ds
\right)^{\frac{{\aa}-2}{2({\aa}-1)}} \,
\left( \int_0^t \|\Dd u (s) \|_2^2 \; ds \right)^{\frac{\aa}{2({\aa}-1)}},  \label{EST}
\end{eqnarray}
for all $t \in [0, \mathcal{T}^*).$
Alteratively, integrating (\ref{EST-NN-1-2}) and using H\"{o}lder inequality and applying (\ref{K-1}) we get a different estimate
\begin{eqnarray}
&&\hskip-.308in
\|\nabla_h u (t)\|_2^2 + \nu \int_0^t \|\nabla_h \nabla u (s) \|_2^2  \; ds
\leq  \|\nabla_h u_0\|_2^2 + C  \int_0^t \left( \|\pp_{x_3} u_3(s)\|_{\aa}^{\frac{2\aa}{\aa-2}}  \|\nabla u (s)\|_2^2 \right) \; ds
\label{est}
\end{eqnarray}
for all $t \in [0, \mathcal{T}^*).$


\subsection{$\|\nabla u\|_2$ estimates}
Taking the  inner product of the  equation (\ref{EQ-1}) with $-\Dd u$ in $L^2$, we obtain
\begin{eqnarray}
&&\hskip-.22in \frac{1}{2} \frac{d \left\|\nabla u \right\|_2^2  }{d t} + \nu
 \left\|\Dd u \right\|_2^2   \nonumber  \\
&&\hskip-.2in = \int_{\mathbb{R}^3} \left[  (u \cdot \nabla) u
\right] \cdot \Dd_h u  \; dx_1dx_2dx_3 +  \int_{\mathbb{R}^3} \left[  (u \cdot \nabla) u
\right] \cdot \pp_{x_3}^2 u  \; dx_1dx_2dx_3    \nonumber   \\
&&\hskip-.2in \leq C \int_{\mathbb{R}^3} \left[  |u_3| \; |\nabla u| \; \left|\nabla_h \nabla u \right|
+ |\nabla_h u|\;
\left|\pp_{x_3} u \right|^2 \right]
  \; dx_1dx_2dx_3.   \label{H1}
\end{eqnarray}
Applying the Cauchy--Schwarz inequality and  (\ref{SI1}), with $r=4$, we obtain
\begin{eqnarray}
 &&\hskip-.2in
 \int_{\mathbb{R}^3}  |\nabla_h u|\;
\left|\pp_{x_3} u \right|^2
  \; dx_1dx_2dx_3
   \leq   \|\nabla_h u \|_2 \; \|\nabla u\|_4^2
   \nonumber \\
  &&\hskip-.2in
  \leq \|\nabla_h u \|_2 \; \|\nabla u\|_2^{1/2}  \|\nabla_h \nabla u\|_2 \;\|\Dd u\|_2^{1/2}.
  \label{EST-NL-1}
\end{eqnarray}
Now, we are ready to complete our proof for the case when condition (\ref{CONN-1}) holds.
 By (\ref{EST-NN-1}) and Young's inequality,  we have
\begin{eqnarray*}
 &&\hskip-.2in
 \int_{\mathbb{R}^3}  |u_3| \; |\nabla u| \; \left|\nabla_h \nabla u \right|
  \; dx_1dx_2dx_3   \\
 &&\hskip-.2in
   \leq   C \| u_3\|_{2}^2 \| \pp_{x_1} u_3\|_{\aa}^{\frac{\aa}{\aa-1}}
\|  \nabla u\|_{2}^{\frac{\aa-2}{\aa-1}} \;
\|\pp_{x_3} \nabla u \|_2^{\frac{\aa}{\aa-1}} + \frac{\nu}{2}\; \| \nabla \nabla_h u \|_{2}^2  \\
 &&\hskip-.2in
   \leq   C \| u_3\|_{2}^{\frac{4(\aa-1)}{\aa-2}}  \| \pp_{x_1} u_3\|_{\aa}^{\frac{2\aa}{\aa-2}}
\|  \nabla u \|_{2}^{2} + \frac{3\nu}{4}\; \| \Dd u \|_{2}^2.
\end{eqnarray*}
As a result of the above and (\ref{H1}), we get
\begin{eqnarray*}
&&\hskip-.4in
\frac{d \left\|\nabla u \right\|_2^2  }{d t} + \frac{\nu}{2}
 \left\|\Dd u \right\|_2^2    \\
 &&\hskip-.35in
\leq C \|\nabla_h u \|_2 \; \|\nabla u\|_2^{1/2}  \|\nabla_h \nabla u\|_2 \;\|\Dd u\|_2^{1/2}
+ C   \| u_3\|_{2}^{\frac{4(\aa-1)}{\aa-2}}  \| \pp_{x_1} u_3\|_{\aa}^{\frac{2\aa}{\aa-2}}
\|  \nabla u \|_{2}^{2};
\end{eqnarray*}
Integrating the above inequality and using H\"{o}lder inequality we obtain
\begin{eqnarray*}
&&\hskip-.268in
\left\|\nabla u(t) \right\|_2^2  + \frac{\nu}{2} \int_0^t \left\|\Dd u\right\|_2^2 \; ds  \\
&&\hskip-.2in
\leq  \left\|\nabla u(0) \right\|_2^2 +  C
 \left(\sup_{0\leq s \leq t}\|\nabla_h u (s)\|_2\right) \;
 \left( \int_0^t \left\|\nabla u\right\|_2^2  \; ds \right)^{\frac{1}{4}}
\left( \int_0^t  \|\nabla_h \nabla u (s) \|_2^2 \; ds \right)^{\frac{1}{2}} \;
\left( \int_0^t \left\|\Dd u (s) \right\|_2^2 \; ds \right)^{\frac{1}{4}} \\
 &&\hskip-.2in
+ C \left(\sup_{0\leq s \leq t}\| u (s)\|_2^{\frac{4(\aa-1)}{\aa-2}}\right)
\left( \int_0^t  \|\pp_{x_3} u_3 (s)\|_{\aa}^{\frac{2\aa}{\aa - 2}}
 \|\nabla u(s)\|_2^{2} \; ds\right).
\end{eqnarray*}
Thanks to (\ref{K-1}) and (\ref{EST}),  we get
\begin{eqnarray*}
&&\hskip-.268in
\left\|\nabla u(t) \right\|_2^2 + \frac{\nu}{2} \int_0^t \left\|\Dd u\right\|_2^2 \; ds  \\
&&\hskip-.2in
\leq  \left\|\nabla u(0) \right\|_2^2 + C K_1^{1/4}\;
\left[  \|\nabla_h u_0\|_2^2 + C  \left( \int_0^t\|\pp_{x_1} u_3(s)\|_{\aa}^{\frac{2{\aa}}{{\aa}-2}}  \|\nabla u (s)\|_2^2\; ds
\right)^{\frac{{\aa}-2}{2({\aa}-1)}} \,
\left( \int_0^t \|\Dd u (s) \|_2^2 \; ds \right)^{\frac{\aa}{2({\aa}-1)}+\frac{1}{4}}  \right]  \\
 &&\hskip-.16in
+
C K_1^{\frac{2(\aa-1)}{\aa-2}}
\left( \int_0^t  \|\pp_{x_1} u_3 (s)\|_{\aa}^{\frac{2\aa}{\aa - 2}}
 \|\nabla u(s)\|_2^{2} \; ds\right).
\end{eqnarray*}
By Young's and H\"{o}lder inequalities, we obtain
\begin{eqnarray}
&&\hskip-.268in
\left\|\nabla u(t) \right\|_2^2 +\frac{\nu}{4}  \int_0^t \left\|\Dd u\right\|_2^2 \; ds  \nonumber   \\
  &&\hskip-.268in
\leq
 C \|\nabla u(0) \|_2^2 + C  \left( \int_0^t\|\pp_{x_1} u_3(s)\|_{\aa}^{\frac{4{\aa}}{{\aa}-3}}  \|\nabla u (s)\|_2^2\; ds
\right)
 \left( \int_0^t  \|\nabla u (s)\|_2^2\; ds
\right)^{\frac{{\aa}-1}{2({\aa}-2)}}   \nonumber
\\
 &&\hskip-.16in
+
C K_1^{\frac{2(\aa-1)}{\aa-2}}
\left( \int_0^t  \|\pp_{x_1} u_3 (s)\|_{\aa}^{\frac{2\aa}{\aa - 2}}
 \|\nabla u(s)\|_2^{2} \; ds\right).   \label{RRR-1}
\end{eqnarray}
Thanks again to (\ref{K-1}), we get
\begin{eqnarray}
&&\hskip-.268in
\left\|\nabla u(t) \right\|_2^2 +\frac{\nu}{4}  \int_0^t \left\|\Dd u\right\|_2^2 \; ds
\leq
 C \|\nabla u(0) \|_2^2 + C   \int_0^t\|\pp_{x_1} u_3(s)\|_{\aa}^{\frac{4{\aa}}{{\aa}-3}}  \|\nabla u (s)\|_2^2\; ds.   \label{RRR-1}
\end{eqnarray}
Therefore, in case (\ref{CON-1}) holds, we apply Gronwall inequality to obtain
\begin{eqnarray*}
&&\hskip-.268in
\left\|\nabla u(t) \right\|_2^2 +\nu \int_0^t \left\|\Dd u\right\|_2^2 \; ds
\leq  C(1+\left\|\nabla u(0) \right\|_2^2 )e^{C\, M},
\end{eqnarray*}
for all $t\in [0, \mathcal{T}^*).$  Therefore, if the condition (\ref{CONN-1}) holds the $H^1$ norm of the
solution $u$ is bounded, and this completes our proof in this case. Next, we complete the proof when
 $u_3$ satisfies (\ref{CONN-2}).
Thanks to (\ref{H1}), (\ref{EST-NN-2}) and (\ref{EST-NL-1}), we get
\begin{eqnarray*}
&&\hskip-.4in
\frac{d \left\|\nabla u \right\|_2^2  }{d t} + \frac{\nu}{2}
 \left\|\Dd u \right\|_2^2    \\
 &&\hskip-.35in
\leq C \|\nabla_h u \|_2 \; \|\nabla u\|_2^{1/2}  \|\nabla_h \nabla u\|_2 \;\|\Dd u\|_2^{1/2}
+ C   \| u_3\|_{2}^{\frac{4(\aa-1)}{\aa-2}}  \| \pp_{x_3} u_3\|_{\aa}^{\frac{2\aa}{\aa-2}}
\|  \nabla u \|_{2}^{2};
\end{eqnarray*}
Integrating the above inequality and using H\"{o}lder inequality we obtain
\begin{eqnarray*}
&&\hskip-.268in
\left\|\nabla u(t) \right\|_2^2  + \frac{\nu}{2} \int_0^t \left\|\Dd u\right\|_2^2 \; ds  \\
&&\hskip-.2in
\leq  \left\|\nabla u(0) \right\|_2^2 +  C
 \left(\sup_{0\leq s \leq t}\|\nabla_h u (s)\|_2\right) \;
 \left( \int_0^t \left\|\nabla u\right\|_2^2  \; ds \right)^{\frac{1}{4}}
\left( \int_0^t  \|\nabla_h \nabla u (s) \|_2^2 \; ds \right)^{\frac{1}{2}} \;
\left( \int_0^t \left\|\Dd u (s) \right\|_2^2 \; ds \right)^{\frac{1}{4}} \\
 &&\hskip-.2in
+ C \left(\sup_{0\leq s \leq t}\| u (s)\|_2^{\frac{4(\aa-1)}{\aa-2}}\right)
\left( \int_0^t  \|\pp_{x_3} u_3 (s)\|_{\aa}^{\frac{2\aa}{\aa - 2}}
 \|\nabla u(s)\|_2^{2} \; ds\right).
\end{eqnarray*}
Thanks to (\ref{K-1}) and (\ref{est}),  we get
\begin{eqnarray*}
&&\hskip-.268in
\left\|\nabla u(t) \right\|_2^2 + \frac{\nu}{2} \int_0^t \left\|\Dd u\right\|_2^2 \; ds  \\
&&\hskip-.2in
\leq  \left\|\nabla u(0) \right\|_2^2 + C K_1^{1/4}\;
\left[  \|\nabla_h u_0\|_2^2 + C  \left( \int_0^t\|\pp_{x_3} u_3(s)\|_{\aa}^{\frac{2{\aa}}{{\aa}-2}}  \|\nabla u (s)\|_2^2\; ds
\right) \,
\left( \int_0^t \|\Dd u (s) \|_2^2 \; ds \right)^{\frac{1}{4}}  \right]  \\
 &&\hskip-.16in
+
C K_1^{\frac{2(\aa-1)}{\aa-2}}
\left( \int_0^t  \|\pp_{x_1} u_3 (s)\|_{\aa}^{\frac{2\aa}{\aa - 2}}
 \|\nabla u(s)\|_2^{2} \; ds\right).
\end{eqnarray*}
By Young's and H\"{o}lder inequalities, we obtain
\begin{eqnarray}
&&\hskip-.268in
\left\|\nabla u(t) \right\|_2^2 +\frac{\nu}{4} \int_0^t \left\|\Dd u\right\|_2^2 \; ds     \nonumber  \\
  &&\hskip-.268in
\leq
 C \|\nabla u(0) \|_2^2 + C  \left( \int_0^t\|\pp_{x_3} u_3(s)\|_{\aa}^{\frac{8{\aa}}{3({\aa}-2)}}  \|\nabla u (s)\|_2^2\; ds
\right)
 \left( \int_0^t  \|\nabla u (s)\|_2^2\; ds
\right)^{\frac{1}{4}}   \nonumber
\\
 &&\hskip-.16in
+
C K_1^{\frac{2(\aa-1)}{\aa-2}}
\left( \int_0^t  \|\pp_{x_3} u_3 (s)\|_{\aa}^{\frac{2\aa}{\aa - 2}}
 \|\nabla u(s)\|_2^{2} \; ds\right).   \label{RRR-2}
\end{eqnarray}
Thanks again to (\ref{K-1}),  we get
\begin{eqnarray}
&&\hskip-.268in
\left\|\nabla u(t) \right\|_2^2 +\frac{\nu}{4} \int_0^t \left\|\Dd u\right\|_2^2 \; ds
\leq
 C \|\nabla u(0) \|_2^2 + C  \left( \int_0^t\|\pp_{x_3} u_3(s)\|_{\aa}^{\frac{8{\aa}}{3({\aa}-2)}}  \|\nabla u (s)\|_2^2\; ds
\right).   \label{RRR-2}
\end{eqnarray}
Therefore, by Gronwall inequality and (\ref{CONN-2}) we obtain
\begin{eqnarray*}
&&\hskip-.268in
\left\|\nabla u(t) \right\|_2^2 +\nu \int_0^t \left\|\Dd u\right\|_2^2 \; ds
\leq  C(1+\left\|\nabla u(0) \right\|_2^2 )e^{C\,M}.
\end{eqnarray*}
for all $t \in [0, \mathcal{T}^*).$ Therefore, the $H^1$ norm of the
strong solution $u$ is bounded on the maximal interval of existence
 $[0, \mathcal{T}^*)$.  This
 completes the proof of Theorem \ref{T-MAIN}.

\noindent
\section*{Acknowledgements}
 This work was
supported in part by the NSF grants no.~DMS-0709228
and no.~DMS-0708832, and by the Alexander von
Humboldt Stiftung/Foundation (E.S.T.). The authors
are also thankful to the kind and warm hospitality
of  the Institute for Mathematics and its
Applications (IMA), University of Minnesota, where
part of this work was completed.

\end{document}